\documentclass{article}

\usepackage{amsmath}
\usepackage{amssymb}
\usepackage{graphicx}
\usepackage{bm}
\usepackage{mathdots}
\usepackage{booktabs}
\usepackage{rotating}
\usepackage[ruled,linesnumbered]{algorithm2e}
\usepackage[a4paper,left=2.8cm,right=2.8cm,top=2.5cm,bottom=2.5cm]{geometry}
\usepackage{fancyhdr}
\usepackage{hyperref}
\usepackage{url}
\usepackage{mathtools}
\usepackage{multirow}

\usepackage[framemethod=tikz]{mdframed}

\newtheorem{theorem}{Theorem}[section]

\newtheorem{lemma}{Lemma}[section]
\newtheorem{definition}{Definition}[section]

\newtheorem{remark}{Remark}[section]

\makeatletter 

\numberwithin{equation}{section}
\makeatother  

\newenvironment{proof}{{\noindent\it Proof.}\quad}{\hfill $\square$\\}

\allowdisplaybreaks

\usepackage[mathscr]{euscript}
\def\mathcal{\mathscr}

\newcommand{\lassohyper}{\mathscr{L}_{L}^{\lambda}}
\newcommand{\filteredhyper}{\mathscr{F}_{L,N}}
\newcommand{\hyper}{\mathscr{L}_{L}}
\newcommand{\hybridhyper}{\mathscr{H}_{L}^{\lambda}}

\newcommand{\discreteinner}[1]{\left\langle#1\right\rangle_N}
\newcommand{\inner}[1]{\left\langle#1\right\rangle}

\newcommand{\squarebrackets}[1]{\left[#1\right]}
\newcommand{\braces}[1]{\left\{#1\right\}}

\newcommand{\softoperatornew}[1]{\mathscr{S}_{\lambda\mu_{\ell}}\left(#1\right)}

\newcommand{\measure}{\text{d}\omega}
\newcommand{\rmd}{\text{d}}

\newcommand{\summ}[2]{\sum\limits_{#1 =1}^{#2}}
\newcommand{\norminfty}[1]{\|#1\|_{\infty}}
\newcommand{\normtwo}[1]{\|#1\|_2}

\begin{document}
\title{Hybrid hyperinterpolation over general regions}

\author{Congpei An\footnotemark[1],\footnotemark[2]  Jiashu Ran\footnotemark[1] and Alvise Sommariva\footnotemark[3]}

\renewcommand{\thefootnote}{\fnsymbol{footnote}}
\footnotetext[1]{School of Mathematics, Southwestern University of Finance and Economics, Chengdu, China (ranjiashuu@gmail.com),}
\footnotetext[2]{Present address:Department of Mathematics, The University of Hong Kong, Hong Kong, China(andbachcp@gmail.com),}
\footnotetext[3]{Department of Mathematics, University of Padova, Italy (alvise@math.unipd.it).}


\maketitle

\begin{abstract}
We present an  $\ell^2_2+\ell_1$-regularized discrete least squares approximation over general regions under assumptions of hyperinterpolation, named hybrid hyperinterpolation. Hybrid hyperinterpolation, using a soft thresholding operator and a filter function to shrink the Fourier coefficients approximated by a high-order quadrature rule of a given continuous function with respect to some orthonormal basis, is a combination of Lasso and filtered hyperinterpolations. Hybrid hyperinterpolation inherits features of them to deal with noisy data once the regularization parameter and the filter function are chosen well.  We derive $L_2$ errors in theoretical analysis for hybrid hyperinterpolation to approximate continuous functions with noise data on sampling points.  Numerical examples illustrate the theoretical results and show that well chosen regularization parameters can enhance the approximation quality over the unit-sphere and the union of disks.
\end{abstract}

\textbf{Keywords: }{ hyperinterpolation, lasso, filter, noise, error estimate}

\textbf{AMS subject classifications.} 65D15, 65A05, 41A10,  33C5

\section{Introduction}
Hyperinterpolation, originally introduced by Sloan in the seminal paper \cite{sloan1995polynomial}, is a constructive approximation method for continuous functions over some compact subsets or manifolds. Briefly speaking, hyperinterpolation uses a high-order quadrature rule to approximate a truncated Fourier expansion in a series of orthogonal polynomials for some measure on a given multidimensional domain, and is the unique solution to a discrete least squares approximation, which means it is also a projection (\cite[Lemma 5]{sloan1995polynomial}). It is well known that the least squares approximation usually tackles with the scenario that the noise level is relatively small \cite{Hesse2021local,hesse2017rbfs,Le2009localized}. Adding a regularization term to the least squares approximation becomes a popular approach to solve the case where the noise level is large. For example, Tikhonov regularized least squares approximation \cite{an2021tikhonov} performs well on $[-1,1]$ to handle noisy data. On the sphere, filtered hyperinterpolation with a certain filter function is equivalent to Tikhonov regularized least squares approximation \cite{an2012regularized,an2021tikhonov}, which is also a constructive polynomial scheme \cite{an2012regularized,sloan2012filtered}. Tikhonov regularized least squares approximation continuously shrinks all coefficients of hyperinterpolation without dismissing less relevant ones, while filtered hyperinterpolation processes these coefficients in a similar way but dismisses some less relevant coefficients. Concisely, they do not possess the basis selection ability.

Lasso hyperinterpolation introduced by An and Wu \cite{an2021lasso} is an $\ell_1$-regularized discrete least square approximation in order to recover functions with noise, which does both continuous shrinkage and automatic basis selection simultaneously. However, Lasso hyperinterpolation could not shrink coefficients in a more flexible manner, that is keeping high relevant coefficients invariant and doubly shrinking less relevant ones. Penalty parameters in Lasso hyperinterpolation are difficult to take effect since they combine with the regularization parameter to shrink coefficients making things more complicated.

In this paper, we develop a new variant of hyperinterpolation based on  filtered hyperinterpolation and Lasso hyperinterpolation --- \emph{hybrid hyperinterpolation}, which actually corresponds to an $\ell^2_2 + \ell_1$-regularized discrete least squares approximation (see Theorem \ref{thm:HybridHyperSol}). Hybrid hyperinterpolation processes coefficients of hyperinterpolation by means of a soft thresholding operator and a filter function. It not only inherits denosing and basis selection abilities of Lasso hyperinterpolation, but also possesses characteristic of filtered hyperinterpolation remaining high relevant coefficients of hyperinterpolation invariant and shrinking less relevant ones; it performs better in function recoveries than Lasso hyperinterpolation and filtered hyperinterpolation over some general regions, and is more robust in denosing. The $L_2$ operator norm of the hybrid hyperinterpolation operator is analyzed, providing a theoretical estimate of the $L_2$ error bound.  

An interesting numerical example over the union of disks (see section \ref{UnionDisk}) is considered, in which the  polynomial orthonormal basis can not be given explicitly by a simply expression and is obtained in numerical ways. This example shows that variants of hyperinterpolation could also contribute to denosing over some more complicated compact regions like it does on the interval, the unit-disk, the unit-sphere and the unit-cube \cite{an2021lasso}. 

The rest of the paper is organized as follows. In section 2, the backgrounds of hyperinterpolation and its variants are reviewed. In section 3, we show that how hyperinterpolation combines a soft thresholding operator with a filter function to become the hybrid hyperinterpolation. $L_2$ errors analysis for hybrid hyperinterpolation to approximate continuous functions with noise and noise-free are derived in section 4. Numerical examples over the unit-sphere and the union of disks are specifically presented in section 5.
\section{Backgrounds of hyperinterpolation}
Let $\Omega$ be a compact set of $\mathbb{R}^s$ with a positive measure $\omega$.  Suppose $\Omega$ has finite measure with respect to $\measure$, that is,
\begin{eqnarray*}
    \int_{\Omega} \rmd \omega = V < \infty.
\end{eqnarray*}

We denote by $L^2(\Omega)$ the Hilbert space of square-integrable functions on $\Omega$ with the $L^2$ inner product
\begin{eqnarray}\label{equ:continuousinner}
\langle f ,g \rangle = \int_{\Omega} fg \rmd \omega , \qquad \forall f,g \in L^2(\Omega),
\end{eqnarray}
and the induced norm $\|f\|_2: = \langle f,f\rangle^{1/2}$. Let $\mathbb{P}_L (\Omega) \subset L^2(\Omega)$ be the linear space of polynomials with total degree not strictly greater than $L$, restricted to $\Omega$, and let $d= \dim (\mathbb{P}_L(\Omega))$ be the dimension of $\mathbb{P}_L(\Omega)$.

Next we define an orthonormal basis of $\mathbb{P}_L(\Omega)$
\begin{eqnarray}\label{set:orthonormalbasis}
\{ \Phi_{\ell}| \ell = 1, \ldots, d \} \subset \mathbb{P}_L(\Omega)
\end{eqnarray}
in the sense of
\begin{eqnarray}\label{equ:kronecker}
\langle \Phi_{\ell},\Phi_{\ell'} \rangle = \delta_{\ell \ell'}, \qquad \forall 1 \leq \ell, \ell' \leq d.
\end{eqnarray}

The $L^2(\Omega)$-\emph{orthogonal projection} $\mathscr{T}_L: L^2(\Omega) \to \mathbb{P}_L(\Omega)$ can be uniquely defined by
\begin{eqnarray}\label{def:L2Projection}
\mathscr{T}_L  f:= \sum_{\ell=1}^{d} \hat{f}_{\ell} \Phi_{\ell} = \sum_{\ell=1}^{d} \langle f, \Phi_{\ell} \rangle \Phi_{\ell}  , \qquad \forall f \in L^2(\Omega),
\end{eqnarray}
where $\{\hat{f}_{\ell}\}_{\ell=1}^{d}$ are the Fourier coefficients
\begin{eqnarray*}
    \hat{f}_{\ell} := \langle f, \Phi_{\ell} \rangle = \int_{\Omega} f\Phi_{\ell} \rmd \omega, \qquad \forall \ell = 1,\ldots,d.
\end{eqnarray*}

In order to numerically evaluate the scalar product in \eqref{def:L2Projection}, it is fundamental to consider an $N$-point quadrature rule of PI-type (Positive weights and Interior nodes), i.e.,
\begin{eqnarray}\label{equ:Npointquad}
\sum_{j=1}^{N} w_{j}g(\mathbf{x}_j) \approx \int_{\Omega} g \rmd \omega, \qquad \forall g \in \mathscr{C}(\Omega),
\end{eqnarray}
where the quadrature points $\{\mathbf{x}_1, \ldots, \mathbf{x}_N\}$ belong to $\Omega$ and the corresponding quadrature weights $\{w_1, \ldots, w_N\}$ are positive, and $\mathcal{C}(\Omega)$ is a continuous function space. Furthermore we say that \eqref{equ:Npointquad} has algebraic degree of exactness $\delta$ if

\begin{eqnarray}\label{equ:exactness}
\sum_{j=1}^{N} w_j p(\mathbf{x}_j) = \int_{\Omega}p \rmd \omega ,   \qquad \forall p \in\mathbb{P}_{\delta}(\Omega).
\end{eqnarray}
With the help of a quadrature rule \eqref{equ:Npointquad} with algebraic degree of exactness $\delta=2L$, we can introduce a ``\emph{discrete inner product}'' on  $\mathscr{C}(\Omega)$ \cite{sloan1995polynomial} by
\begin{eqnarray}\label{equ:discreteinner}
\langle f,g \rangle_N := \sum_{j=1}^{N} w_j f(\mathbf{x}_j) g(\mathbf{x}_j), \qquad \forall f,g \in \mathscr{C}(\Omega),
\end{eqnarray}
corresponding to the $L^2(\Omega)$-inner product \eqref{equ:continuousinner}. For any $p,q \in \mathbb{P}_{L}(\Omega)$, the product $pq$ is a polynomial in $\mathbb{P}_{2L}(\Omega)$. Therefore it follows from the quadrature exactness of \eqref{equ:exactness} for polynomials of degree at most $2L$ that
\begin{eqnarray}\label{equ:kroneckerdiscrete}
\langle {p,q} \rangle_N = \langle {p,q} \rangle = \int_{\Omega} pq \rmd \omega, \qquad \forall p, q \in \mathbb{P}_{L}(\Omega).
\end{eqnarray}

In 1995, Sloan introduced in \cite{sloan1995polynomial} the \emph{hyperinterpolation operator} $\mathscr{L}_L: \mathscr{C}(\Omega) \to \mathbb{P}_L(\Omega)$ as
\begin{eqnarray}\label{equ:hyper}
\mathcal{L}_L f:= \sum_{\ell=1}^{d} \langle {f, \Phi_{\ell}} \rangle_N   \Phi_{\ell}, \qquad \forall  f \in \mathscr{C}(\Omega).
\end{eqnarray}
$\hyper{f}$ is a projection of $f$ onto $\mathbb{P}_L(\Omega)$ obtained by replacing the $L^2(\Omega)$-inner products \eqref{equ:continuousinner} in the $L^2(\Omega)$-orthogonal projection $\mathscr{T}_L f$ by the discrete inner products \eqref{equ:discreteinner}.

To explore one of its important features, we introduce the following discrete least squares approximation problem
\begin{eqnarray}\label{HyperLeastApprox}
\min_{p \in \mathbb{P}_L(\Omega) }  \left\{ \frac{1}{2} \sum_{j=1}^{N}w_j[p(\mathbf{x}_j) - f(\mathbf{x}_j)]^2 \right\}
\end{eqnarray}
with $p(\mathbf{x}) = \sum_{\ell=1}^{d}\alpha_{\ell}\Phi_{\ell}(\mathbf{x}) \in \mathbb{P}_{L}(\Omega)$, or equivalently
\begin{eqnarray}\label{HyperLeastApproxMatrix}
\min\limits_{\bm{\alpha} \in \mathbb{R}^d}~~ \frac{1}{2} \|\mathbf{W}^{1/2} ( \mathbf{A}\bm{\alpha} - \mathbf{f})\|_2^2,
\end{eqnarray}
where $\mathbf{W}=\text{diag}(w_1,\ldots,w_N)$ is the quadrature weights matrix, $\mathbf{A}=(\Phi_{\ell}(\mathbf{x}_j))\in\mathbb{R}^{N\times d}$ is the sampling matrix, $\bm{\alpha}=[\alpha_1,\ldots,\alpha_d]^{\text{T}}\in\mathbb{R}^d$ and $\mathbf{f}=[f(\mathbf{x}_1),\ldots,f(\mathbf{x}_N)]^{\text{T}}\in\mathbb{R}^N$ are two column vectors (recall $\mathbf{x}_j\in\mathbb{R}^s$). Sloan in \cite{sloan1995polynomial}  revealed that the relation between the hyperinterpolant $\hyper{f}$ and the best discrete least squares approximation (weighted by quadrature weights) of $f$ at the quadrature points. More precisely he proved the following important result:
\begin{lemma}[Lemma 5 in \cite{sloan1995polynomial}]\label{prop:HyperLeastApprox}
Given $f\in\mathscr{C}(\Omega)$, let $\hyper{f}\in\mathbb{P}_L(\Omega)$ be defined by \eqref{equ:hyper}, where the discrete scalar product $\langle f,g \rangle_N$ in (\ref{equ:discreteinner}) is defined by an $N$-point quadrature rule of PI-type in $\Omega$ with algebraic degree of exactness $2L$. Then $\hyper{f}$ is the unique solution to the approximation problem \eqref{HyperLeastApprox}.
\end{lemma}


\subsection{Filtered hyperinterpolation}\label{sec:filteredhyperinterpolation}
Filtered hyperinterpolation was introduced by Sloan and Womersley on the unit sphere \cite{sloan2012filtered}, all coefficients of the hyperinterpolant are filtered by a filter function $h(x)$. To this purpose, we introduce a {\it{filter function}} $h \in \mathcal{C}[0,+\infty)$ that satisfies
\begin{eqnarray*}
h(x)=\begin{cases}
1, &\text{ for }x\in[0,1/2],\\
0, &\text{ for }x\in[1,\infty).
\end{cases}
\end{eqnarray*}
Depending on the behaviour of $h$ in $[1/2,1]$, one can define many filters. Many examples were described in \cite{sloan2011polynomial}.  In this paper, we only consider the filtered function $h(x)$ defined by
\begin{eqnarray}\label{equ:filtered}
h(x)=\begin{cases}
    1, & x \in [0,\frac{1}{2}] ,\\
\sin^2(\pi x), & x \in [\frac{1}{2},1].
\end{cases}
\end{eqnarray}

In what follows, we denote by $\lfloor \cdot \rfloor$ the floor function. Once a filter has been chosen, one can introduce the filtered hyperinterpolation as follows:

\begin{definition}[\cite{sloan2012filtered}]
Suppose that the discrete scalar product $\langle f,g \rangle_N$ in \eqref{equ:discreteinner} is determined by an $N$-point quadrature rule of PI-type in $\Omega$ with algebraic degree of exactness  $L-1+\lfloor L/2 \rfloor$. The \emph{filtered hyperinterpolant} $\mathcal{F}_{L,N}{f} \in \mathbb{P}_{L-1}(\Omega)$ of $f$ is defined as
\begin{eqnarray}\label{equ:filteredhyper}
    \mathcal{F}_{L,N}{f}:=\sum_{\ell=1}^{d} h\left(\frac{\deg \Phi_{\ell}}{L}\right) \langle f,\Phi_{\ell} \rangle_N \Phi_{\ell},
\end{eqnarray}
where $d={\mbox{dim}}({\mathbb{P}}_L(\Omega))$ and $h$ is a filter function defined by \eqref{equ:filtered}.
\end{definition}

For simplicity, we use $h_{\ell}$ instead of $h(\deg \Phi_{\ell}/L)$ in this paper. Notice that by \eqref{equ:filteredhyper}, depending on the fact $\overline{\text{supp} (h)}=[0,1]$, with respect to the classical hyperinterpolation, one can achieve more {\it{sparsity}} in the polynomial coefficients in view of the term $h_{\ell}$, i.e. some less relevant discrete Fourier coefficients are dismissed. Indeed, if $\deg \Phi_{\ell} \geq L$ then $h_{\ell} =0$.

In view of the assumption on the algebraic degree of precision, from \eqref{equ:filteredhyper}, we easily have $\filteredhyper{f}=f$ for all $f\in\mathbb{P}_{\lfloor L/2\rfloor}(\Omega)$.

Recently, it was shown in \cite{lin2021distributed,Montúfar2022distributed} that (distributed) filtered hyperinterpolation can reduce weak noise on spherical functions.

\subsection{Lasso hyperinterpolation}\label{sec:lassohyperinterpolation}
An alternative to filtered hyperinterpolation is the so-called {\it{Lasso hyperinterpolation}} \cite{an2021lasso}, which aims to denoise and perform feature selection by discarding less relevant discrete Fourier coefficients. Lasso hyperinterpolation employs a soft thresholding operator, defined as follows, to select hyperinterpolation coefficients whose absolute values do not exceed predetermined threshold values.
\begin{definition}[\cite{donoho1994ideal}]\label{def:soft}
The {\it{soft thresholding operator}} is defined as
\begin{eqnarray}\label{equ:soft}
    \mathcal{S}_{k}(a)=\begin{cases}
        a+k, & {\mbox{ if }} a < -k ,\\
0,  & {\mbox{ if }} -k \leq a \leq k, \\
a-k, & {\mbox{ if }}  a > k, \\
    \end{cases}
\end{eqnarray}
where $k \geq 0$.
\end{definition}

Suppose that the discrete scalar product $\langle f,g \rangle_N$ in \eqref{equ:discreteinner} is determined by an $N$-point quadrature rule of PI-type in $\Omega$ with algebraic degree of exactness $2L$.
Then the \emph{Lasso hyperinterpolation} of $f$ onto $\mathbb{P}_{L}(\Omega)$ is defined as
\begin{eqnarray}\label{equ:lassohyper}
    \mathcal{L}_L^{\lambda} f:= \sum_{\ell=1}^{d} {\mathcal{S}}_{\lambda \mu_{\ell}} (\langle f,\Phi_{\ell} \rangle_N) \Phi_{\ell},
\end{eqnarray}
where $\lambda>0$ is the regularization parameter and $\{\mu_{\ell}\}_{\ell=1}^{d}$ is a set of positive penalty parameters. The effect of the soft threshold operator is such that if $|\langle f,\Phi_{\ell} \rangle_N| \leq \lambda \mu_{\ell}$ then ${\mathcal{S}}_{\lambda \mu_{\ell}} (\langle f,\Phi_{\ell} \rangle_N)=0$.

Now we introduce the $\ell_1$-regularized least squares problem
\begin{eqnarray}\label{L1HyperLeastApprox}
  \min\limits_{p_{\lambda} \in \mathbb{P}_L(\Omega)}  \left\{ \frac{1}{2}\sum_{j=1}^{N} w_j[ p_{\lambda}(\mathbf{x}_j) - f(\mathbf{x}_j)]^2  + \lambda\sum_{\ell=1}^{d}\mu_{\ell}|\gamma_{\ell}^{\lambda}| \right\}
\end{eqnarray}
with $p_{\lambda}(\mathbf{x})=\sum_{\ell=1}^d \gamma_{\ell}^{\lambda} \Phi_{\ell}(\mathbf{x}) \in \mathbb{P}_L(\Omega)$, or equivalently
\begin{eqnarray}\label{L1HyperLeastApproxMatrix}
    \min\limits_{\bm{\gamma}^{\lambda} \in \mathbb{R}^d}~~ \frac{1}{2}\|\mathbf{W}^{1/2}(\mathbf{A}\bm{\gamma}^{\lambda} - \mathbf{f}) \|_2^2  + \lambda \| \mathbf{R}_1\bm{\gamma}^{\lambda}\|_1 , \qquad \lambda>0,
\end{eqnarray}
where $\bm{\gamma}^{\lambda}=[\gamma_1^{\lambda},\ldots,\gamma_d^{\lambda}]^{\text{T}}\in\mathbb{R}^d$ and $\mathbf{R}_1 = \text{diag}(\mu_1, \ldots, \mu_d)$. In this paper, the subscript and superscript in $p_{\lambda}$ and $\bm{\gamma}^{\lambda}$ mean that their concrete forms are related to the regularization parameter $\lambda$. The following result holds.
\begin{theorem}[Theorem 3.4 in \cite{an2021lasso}]\label{thm:lassohypersol}
Let $\lassohyper{f} \in \mathbb{P}_L(\Omega)$ be defined by \eqref{equ:lassohyper}, and adopt conditions of Lemma \ref{prop:HyperLeastApprox}. Then $\lassohyper{f}$ is the solution to the regularized least squares approximation problem \eqref{L1HyperLeastApprox}.
\end{theorem}





\section{Hybrid hyperinterpolation}\label{sec:hybridhyperinterpolation}
Here, we introduce a novel approach called  \emph{hybrid hyperinterpolation}, i.e., a composition of filtered hyperinterpolation and Lasso hyperinterpolation.  Specifically, hybrid hyperinterpolation corresponds to solving an $\ell^2_2+\ell_1$-regularized least approximation problem.

In practice, the sampling data $\{f(\mathbf{x}_j)\}_{j=1}^{N}$ at nodes $\{\mathbf{x}_j\}_{j=1}^{N}$ are perturbed by noise $\epsilon$. We usually obtain the data $\{\mathbf{x}_j, f^{\epsilon}(\mathbf{x}_j)\}_{j=1}^{N}$ with $f^{\epsilon}(\mathbf{x}_j)=f(\mathbf{x}_j)+ \epsilon(\mathbf{x}_j)$, and actually solve the following $\ell^2_2+\ell_1$-regularized least approximation problem
\begin{equation}\label{L22L1HyperLeastApprox}
    \min\limits_{p_{\lambda} \in \mathbb{P}_L(\Omega)}  \left\{ \frac{1}{2}\summ{j}{N} w_j[p_{\lambda}(\mathbf{x}_j) - f^{\epsilon}(\mathbf{x}_j)]^2
+ \frac{1}{2}\summ{j}{N} w_j[ \mathcal{R}_{L}{p_{\lambda}(\mathbf{x}_j)}]^2 + \lambda\summ{\ell}{d}\mu_{\ell}|\beta_{\ell}^{\lambda}| \right\},
\end{equation}
where $\lambda >0$, $\mu_1,\ldots,\mu_d >0$ and $p_{\lambda}(\mathbf{x})=\sum_{\ell=1}^d \beta_{\ell}^{\lambda} \Phi_{\ell}(\mathbf{x}) \in \mathbb{P}_L(\Omega)$ and
\begin{eqnarray}\label{R1}
\mathcal{R}_{L}{p_{\lambda}(\mathbf{x})} = \sum_{\ell=1}^{d} b_{\ell} \langle{\Phi_{\ell}, p_{\lambda}} \rangle_N \Phi_{\ell} = \sum_{\ell=1}^{d} b_{\ell} \beta_{\ell}^{\lambda} \Phi_{\ell}(\mathbf{x})
\end{eqnarray}
with
\begin{eqnarray}\label{coef:b}
    b_{\ell} = \begin{cases}
        0, & \frac{\deg \Phi_{\ell}}{L} \in [0,\frac{1}{2}] ,\\
\sqrt{\frac{1}{h_{\ell}} -1}, & \frac{\deg \Phi_{\ell}}{L} \in [\frac{1}{2},1),
    \end{cases}
\end{eqnarray}
in which $h$ is a function suitable for filtered hyperinterpolation defined by \eqref{equ:filtered}. Note that in \eqref{coef:b} we have excluded $\deg \Phi_{\ell} /L=1$ because we would have $b_{\ell}=\infty$ and hence $\beta^{\lambda}_{\ell}=0$ in \eqref{coef:beta} in this case.

It is not difficult to see that \eqref{L22L1HyperLeastApprox} is equivalent to

\begin{eqnarray}\label{L22L1HyperLeastApproxMatrix}
    \min\limits_{\bm{\beta}^{\lambda} \in \mathbb{R}^d}~~ \left\{ \frac{1}{2}\|\mathbf{W}^{1/2}(\mathbf{A}\bm{\beta}^{\lambda} - \mathbf{f}^{\epsilon}) \|_2^2  + \frac{1}{2}\|\mathbf{W}^{1/2}\mathbf{R}_2\bm{\beta}^{\lambda}\|_2^2 + \lambda \| \mathbf{R}_1\bm{\beta}^{\lambda}\|_1  \right\},
\end{eqnarray}
where $\bm{\beta}^{\lambda}=[\beta_1^{\lambda},\ldots,\beta_d^{\lambda}]^{\text{T}}\in\mathbb{R}^d$ and $\mathbf{R}_1 = \text{diag}(\mu_1, \ldots, \mu_d)$, $\mathbf{f}^{\epsilon}=[f^{\epsilon}(\mathbf{x}_1), \ldots, f^{\epsilon}(\mathbf{x}_N)]^{\rm{T}} \in \mathbb{R}^{N}$, $\mathbf{R}_2=\mathbf{AB} \in \mathbb{R}^{N \times d}$ with $\mathbf{A}=(\Phi_{\ell}(\mathbf{x}_j)) \in \mathbb{R}^{N \times d}$ and $\mathbf{B}= \text{diag}(b_{1},\ldots, b_d), \mathbf{W}= \text{diag}(w_{1},\ldots, w_N)$.

 We define the \emph{hybrid hyperinterpolation} as follows.

\begin{definition}[Hybrid hyperinterpolation]
Let $\Omega \subset {\mathbb{R}}^s$  be a compact domain and $f \in \mathcal{C}(\Omega)$.
Suppose that the discrete scalar product $\langle f,g \rangle_N$ in (\ref{equ:discreteinner}) is determined by an $N$-point quadrature rule of PI-type in $\Omega$ with algebraic degree of exactness $2L$. The \emph{hybrid hyperinterpolation} of $f$ onto $\mathbb{P}_{L}(\Omega)$ is defined as
\begin{eqnarray}\label{equ:hybridhyper}
  \mathcal{H}_{L}^{\lambda}{f}:=\sum_{\ell=1}^{d} h\left(\frac{\deg \Phi_{\ell}}{L} \right)\mathcal{S}_{\lambda \mu_{\ell}} (\langle{f, \Phi_{\ell}} \rangle_N)\Phi_{\ell}, \qquad \lambda>0,
\end{eqnarray}
where $h(\cdot)$ is a filter function defined by \eqref{equ:filtered}, $\{\mathcal{S}_{\lambda \mu_{\ell}}(\cdot)\}_{\ell=1}^{d}$ are soft thresholding operators and $\mu_1,\ldots,\mu_d >0$.
\end{definition}

Then we obtain the following important result.
\begin{theorem}\label{thm:HybridHyperSol}
Let $\Omega \subset {\mathbb{R}}^s$ be a compact domain and $f \in \mathcal{C}(\Omega)$.
Suppose that the discrete scalar product $\langle f,g \rangle_N$ in \eqref{equ:discreteinner} is determined by an $N$-point quadrature rule of PI-type in $\Omega$ with algebraic degree of exactness $2L$. Next let $\mu_1,\ldots,\mu_d >0$, ${\cal{R}}_L$ as in \eqref{R1} and $\hybridhyper{f} \in \mathbb{P}_L(\Omega)$ be defined by \eqref{equ:hybridhyper}. Then $\hybridhyper{f}$ is the solution to the regularized least squares approximation problem \eqref{L22L1HyperLeastApprox} with the noise $\epsilon \equiv 0$.
\end{theorem}

\begin{proof}
Since the problem \eqref{HyperLeastApproxMatrix} is a strictly convex problem, the stationary point of the objective in \eqref{HyperLeastApproxMatrix} with respect to $\bm{\alpha}$ leads to the first-order condition
\begin{eqnarray}\label{equ:firstorder}
    \mathbf{A}^{\text{T}}\mathbf{WA}\bm{\alpha}-\mathbf{A}^{\text{T}}\mathbf{Wf}=\mathbf{0}.
\end{eqnarray}
Note that $\mathbf{A}^{\text{T}}\mathbf{WA}=\mathbf{I}_{d}$, where $\mathbf{I}_d \in {\mathbb R}^{d \times d}$ is the identity matrix, since
\begin{equation}\label{equ:identity}
[\mathbf{A}^{\text{T}}\mathbf{WA}]_{ik} =\summ{j}{N} w_j\Phi_{i}(\mathbf{x}_j)\Phi_{k}(\mathbf{x}_j)=\discreteinner{\Phi_i,\Phi_k} =\inner{\Phi_{i},\Phi_{k}}=\delta_{ik}. \quad 1\leq i,k \leq d.
\end{equation}
By \eqref{equ:firstorder} and \eqref{equ:identity}, we have $\bm{\alpha}=\mathbf{A}^{\text{T}}\mathbf{Wf}$. Then $\bm{\beta}^{\lambda}=[\beta_{1}^{\lambda},\ldots,\beta_{d}^{\lambda}]^{\text{T}}$ is a solution to the problem \eqref{L22L1HyperLeastApproxMatrix} if and only if
\begin{eqnarray}\label{equ:exactfirstorder}
   \textbf{0} \in  \mathbf{A}^{\text{T}}\mathbf{WA}\bm{\beta}^{\lambda} - \mathbf{A}^{\text{T}}\mathbf{Wf} + \lambda \partial (\|\mathbf{R}_1\bm{\beta}^{\lambda}\|_1) +  \mathbf{B}^{\text{T}}\mathbf{A}^{\text{T}}\mathbf{W}\mathbf{AB}\bm{\beta}^{\lambda},
\end{eqnarray}
where $\partial(\cdot)$ denotes the subdifferential \cite[Definition 3.2]{beck2017optimization}.

Thus the first-order condition \eqref{equ:exactfirstorder} amounts to $d$ one-dimensional problems
\begin{eqnarray}\label{equ:1dexactfirstorder}
   0 \in \beta_{\ell}^{\lambda} - \alpha_{\ell} + \lambda\mu_{\ell} \partial(|\beta_{\ell}^{\lambda}|) +  b_{\ell}^2\beta_{\ell}^{\lambda}, \quad \forall \ell=1,\ldots, d,
\end{eqnarray}
where
\begin{eqnarray*}
     \partial(|\beta_{\ell}^{\lambda}|)= \begin{cases}
         1, & \text{ if } \beta_{\ell}^{\lambda} > 0, \\
\in [-1,1], & \text{ if } \beta_{\ell}^{\lambda} = 0, \\
-1, & \text{ if } \beta_{\ell}^{\lambda} < 0.
     \end{cases}
\end{eqnarray*}

Let $\bm{\beta}^{\lambda}=[\beta_{1}^{\lambda},\ldots,\beta_{d}^{\lambda}]^{\text{T}}$ be the optimal solution to the problem \eqref{L22L1HyperLeastApproxMatrix}. Then we have
\begin{eqnarray*}
    \beta_{\ell}^{\lambda}= \frac{1}{1+b_{\ell}^2}  [\alpha_{\ell} - \lambda\mu_{\ell} \partial(|\beta_{\ell}^{\lambda}|)], \quad \forall \ell=1,\ldots,d.
\end{eqnarray*}
In fact, there are three cases we need to consider:
\begin{enumerate}
\item[(i)] If $\alpha_{\ell} > \lambda\mu_{\ell}$, then $\alpha_{\ell} - \lambda\mu_{\ell} \partial(|\beta_{\ell}^{\lambda}|)>0$; thus, $\beta_{\ell}^{\lambda} >0$, yielding $\partial(|\beta_{\ell}^{\lambda}|)=1$, and then $\beta_{\ell}^{\lambda} = (\alpha_{\ell} - \lambda\mu_{\ell})/(1+b_{\ell}^2) >0 $.
\item[(ii)] If $\alpha_{\ell} < -\lambda\mu_{\ell}$, then $\alpha_{\ell} - \lambda\mu_{\ell} \partial(|\beta_{\ell}^{\lambda}|)<0$; thus, $\beta_{\ell}^{\lambda} <0$, yielding $\partial(|\beta_{\ell}^{\lambda}|)=-1$, and then $\beta_{\ell}^{\lambda} = (\alpha_{\ell} + \lambda\mu_{\ell})/(1+b_{\ell}^2) <0 $.
\item[(iii)] If $-\lambda\mu_{\ell} \leq \alpha_{\ell} \leq \lambda\mu_{\ell} $, then $\beta_{\ell}^{\lambda}>0$ means that $\beta_{\ell}^{\lambda}= (\alpha_{\ell} - \lambda\mu_{\ell})/(1+b_{\ell}^2) <0 $, and $\beta_{\ell}^{\lambda}<0$ means that $\beta_{\ell}^{\lambda}=(\alpha_{\ell} + \lambda\mu_{\ell})/(1+b_{\ell}^2)>0$. Thus, we get two contradictions and $\beta_{\ell}^{\lambda}=0$.
\end{enumerate}

Recall that $\alpha_{\ell}=\discreteinner{f,\Phi_{\ell}} $ and $b_{\ell}$ defined by \eqref{coef:b} for all $\ell=1,\ldots,d$, we deduce that the polynomial constructed with coefficients
\begin{eqnarray}\label{coef:beta}
  \beta_{\ell}^{\lambda} = \frac{\mathcal{S}_{\lambda\mu_{\ell}} (\langle{f,\Phi_{\ell}} \rangle_N)}{1+b_{\ell}^2} =h_{\ell}\mathcal{S}_{\lambda\mu_{\ell}} (\langle{f,\Phi_{\ell}} \rangle_N)
\end{eqnarray}
is hybrid hyperinterpolation $\mathcal{H}_{L}^{\lambda}{f}$.
\end{proof}

\begin{remark}
It is clear that hybrid hyperinterpolation $\mathcal{H}_{L}^{\lambda} f$ can be regarded as the composite of filtered hyperinterpolation $\mathcal{F}_{L,N} f$ and Lasso hyperinterpolation $\mathcal{L}_{L}^{\lambda} f$, that is,
\begin{eqnarray*}
     \mathcal{H}_{L}^{\lambda} f= \mathcal{F}_{L,N}(\mathcal{L}_{L}^{\lambda}f).
\end{eqnarray*}
Observe that, since we suppose that $h:[0,1] \rightarrow [0,1]$, necessarily
\begin{eqnarray*}
    |h_{\ell} \langle{f,\Phi_{\ell}} \rangle_N| \leq |\langle{f,\Phi_{\ell}} \rangle_N|,
\end{eqnarray*}
the commutative property between $\mathcal{F}_{L,N} f$ and $\mathcal{L}_{L}^{\lambda} f$ does not hold in general, i.e., it may be
\begin{eqnarray*}
    \mathcal{F}_{L,N}(\mathcal{L}_{L}^{\lambda}f) \neq \mathcal{L}_{L}^{\lambda}( \mathcal{F}_{L,N} f) .
\end{eqnarray*}
\end{remark}

\section{Error analysis in the \texorpdfstring{$L^2(\Omega)$}{}-norm} \label{sec:L2NormError}
In the following, the norms which appear in the theorems are defined by
\begin{equation*}
\begin{aligned}
\|g\|_2 &:= \left( \int_\Omega g^2 \text{d}\omega\right)^{1/2}, \quad g \in L^2(\Omega),\\
\|g\|_{\infty} &:= \sup\limits_{\mathbf{x} \in \Omega} |g(\mathbf{x})|,  \qquad \: \:\:g \in \mathcal{C}(\Omega).
\end{aligned}
\end{equation*}
For $g \in \mathcal{C}(\Omega)$, let $p^{\ast} \in \mathbb{P}_{n}(\Omega)$ be the best approximation of $g$ in $\mathbb{P}_{n}(\Omega)$, i.e.,
\begin{equation*}
    E_{n}(g):=\inf_{p \in \mathbb{P}_{n}(\Omega)} \| g-p\|_{\infty} = \| g-p^{\ast}\|_{\infty}.
\end{equation*}
\begin{lemma}[Theorem 1 in \cite{sloan1995polynomial}]\label{lem:HyperErrorL2norm}
Under the conditions of Lemma \ref{prop:HyperLeastApprox}. Then
\begin{eqnarray}\label{equ:stability}
    \| \mathcal{L}_L f \|_2 \leq V^{1/2} \|f \|_{\infty},
\end{eqnarray}
and
\begin{eqnarray}\label{equ:error}
 \| \mathcal{L}_L f -f \|_2 \leq2V^{1/2}E_L(f).
\end{eqnarray}
Thus $ \| \mathcal{L}_L f -f \|_2\to0$ as $L\to\infty$.
\end{lemma}

Under assumptions of Theorem \ref{thm:HybridHyperSol}, let $\alpha_{\ell} = \langle f,\Phi_{\ell} \rangle_N= \sum_{j=1}^{N}w_{j} f(\mathbf{x}_j) \Phi_{\ell}(\mathbf{x}_j)$ and introduce the operator
\begin{eqnarray}\label{ope:Kf}
    K(f):= \sum_{\ell=1}^{d} \{ h_{\ell}\softoperatornew{\alpha_{\ell}}\cdot \alpha_{\ell} - \squarebrackets{h_{\ell}\softoperatornew{\alpha_{\ell}}}^2   \}  \geq 0.
\end{eqnarray}
Note that the positiveness of $K(f)$ is a direct consequence of the fact that
\begin{eqnarray}\label{Kf}
    \begin{cases}
        \alpha_{\ell} < h_{\ell}\softoperatornew{\alpha_{\ell}} \leq 0, \quad {\mbox{if }} \alpha_{\ell} <0, \\
0 \leq h_{\ell} \softoperatornew{\alpha_{\ell}} \leq \alpha_{\ell}, \quad {\mbox{if }} \alpha_{\ell} \geq 0,
    \end{cases}
\end{eqnarray}
which easily implies that each argument of the sum on the right-hand side of \eqref{Kf} is non-negative.  With the help of $K(f)$, the following lemma holds:

\begin{lemma}\label{lem:geomhybrid}
Under the conditions of Theorem \ref{thm:HybridHyperSol},
\begin{enumerate}
\item[(a)] $\discreteinner{f- \hybridhyper{f},\hybridhyper{f}}= K(f)$;
\item[(b)] $\discreteinner{\hybridhyper{f},\hybridhyper{f}} + \discreteinner{f- \hybridhyper{f},f- \hybridhyper{f}} = \discreteinner{f,f}-2K(f)$;
\item[(c)] $K(f) \leq \discreteinner{f,f}/2$;
\item[(d)] $\discreteinner{\hybridhyper{f},\hybridhyper{f}} \leq \discreteinner{f,f} - 2K(f)$.
\end{enumerate}
\end{lemma}

\begin{proof}
(a) We obtain
\begin{equation*}
\begin{aligned}
&\discreteinner{f- \hybridhyper{f},\hybridhyper{f}}\\
=& \discreteinner{f- \summ{\ell}{d} h_{\ell}\softoperatornew{\alpha_{\ell}}\Phi_{\ell}  ,  \summ{k}{d} h_{k}\mathscr{S}_{\lambda \mu_k}(\alpha_{k})\Phi_{k}}  ,\\
=& \summ{k}{d}   h_k\mathscr{S}_{\lambda \mu_k}(\alpha_{k}) \discreteinner{ f- \summ{\ell}{d} h_{\ell}\softoperatornew{\alpha_{\ell}}\Phi_{\ell}, \Phi_k   },
\end{aligned}
\end{equation*}
and
\begin{equation*}
\begin{aligned}
 &\discreteinner{ f- \summ{\ell}{d} h_{\ell}\softoperatornew{\alpha_{\ell}}\Phi_{\ell}, \Phi_k   }\\
 =& \discreteinner{f,\Phi_k}-\discreteinner{\summ{\ell}{d}h_{\ell}\softoperatornew{\alpha_{\ell}}\Phi_{\ell},\Phi_k} ,\\
=& \alpha_k- h_k\mathscr{S}_{\lambda \mu_k}(\alpha_{k})  .
\end{aligned}
\end{equation*}
Hence, we have
\begin{equation*}
\begin{aligned}
&\summ{k}{d} h_k\mathscr{S}_{\lambda \mu_k}(\alpha_{k}) \discreteinner{ f- \summ{\ell}{d} h_{\ell}\softoperatornew{\alpha_{\ell}}\Phi_{\ell}, \Phi_k   } \\
=&\summ{k}{d} h_{k}\mathscr{S}_{\lambda \mu_k}(\alpha_{k}) \cdot \squarebrackets{\alpha_k- h_{k}\mathscr{S}_{\lambda \mu_k}(\alpha_{k})},\\
=&\summ{k}{d} \braces{h_{k}\mathscr{S}_{\lambda \mu_k}(\alpha_{k})\alpha_{k} - \squarebrackets{h_k\mathscr{S}_{\lambda \mu_k}(\alpha_{k})}^2},\\
=&K(f).
\end{aligned}
\end{equation*}

(b) From $\discreteinner{\hybridhyper{f},\hybridhyper{f}}= \discreteinner{f,\hybridhyper{f}}- K(f)$ and
\begin{eqnarray*}
    \discreteinner{f- \hybridhyper{f},f- \hybridhyper{f}} = \discreteinner{f,f} - 2 \discreteinner{f,\hybridhyper{f}} +   \discreteinner{\hybridhyper{f},\hybridhyper{f}},
\end{eqnarray*}
we immediately have
\begin{eqnarray}\label{equ:geomhybrid}
    \discreteinner{\hybridhyper{f},\hybridhyper{f}} + \discreteinner{f- \hybridhyper{f},f- \hybridhyper{f}} = \discreteinner{f,f}-2K(f).
\end{eqnarray}

(c) From equality \eqref{equ:geomhybrid} and the non-negativeness of $ \discreteinner{g,g}$ for any $g \in \mathcal{C}(\Omega)$, we obtain $\discreteinner{f,f}-2K(f) \geq 0$, which means that $ K(f) \leq \discreteinner{f,f}/2 $.

(d) From equality \eqref{equ:geomhybrid} and the non-negativeness of $ \discreteinner{f- \hybridhyper{f},f- \hybridhyper{f}} $, we immediately have $\discreteinner{\hybridhyper{f},\hybridhyper{f}} \leq \discreteinner{f,f} - 2K(f)$.
\end{proof}


\begin{remark}
    Note that when $\lambda=0$ in hybrid hyperinterpolation $\mathcal{H}^{\lambda}_Lf$, we obtain filtered hyperinterpolation $\mathcal{F}_{L,N}f$. And the above results also hold for filtered hyperinterpolation with $\mathcal{H}_{L}^{\lambda}f$ being replaced by $\mathcal{F}_{L,N}f$, in which the soft thresholding operator becomes the identity operator in $K(f)$ in  \eqref{ope:Kf}.
\end{remark}

Next, we examine the $L_2$ operator norm of the hybrid hyperinterpolation operator and provide a theoretical $L_2$ error bound estimate.

\begin{theorem}\label{thm:HybridHyperLeastApproxNoise}
Under the conditions of Theorem \ref{thm:HybridHyperSol}. Further, assume $f^{\epsilon} = f+ \epsilon$ is a noisy version of $f$, where $\epsilon$ is some noise. Then there exists $\tau_1 <1$, which relies on $f$ such that
\begin{eqnarray}\label{norm:hybrid}
\normtwo{\mathcal{H}_{L}^{\lambda}{f}} \leq \tau_1 V^{1/2} \norminfty{f},
\end{eqnarray}
and
\begin{eqnarray}\label{norm:L2ErrHybridNoise}
\normtwo{\hybridhyper{f^{\epsilon}} - f }  \leq  \| \mathcal{H}_{L}^{\lambda} f^{\epsilon} - \mathcal{F}_{L,N}f^{\epsilon} \|_{2} + 2V^{1/2}E_{\lfloor L/2\rfloor}(f) +    V^{1/2} \| f^{\epsilon} -f  \|_{\infty}.
\end{eqnarray}
Thus,
\begin{eqnarray*}
    \lim_{L \to \infty} \lim_{\lambda \to 0} \normtwo{\hybridhyper{f^{\epsilon}} - f } = V^{1/2}\norminfty{f-f^{\epsilon}}.
\end{eqnarray*}
\end{theorem}

\begin{proof}
For any $f \in \mathcal{C}(\Omega)$, we get $\hybridhyper{f} \in \mathbb{P}_L(\Omega)$. By Lemma \ref{lem:geomhybrid} (d), we have
\begin{equation*}
\begin{aligned}
    \normtwo{\hybridhyper{f}}^2&=\inner{\hybridhyper{f},\hybridhyper{f}} = \discreteinner{\hybridhyper{f},\hybridhyper{f}} \leq \discreteinner{f,f} - 2K(f), \\
&= \summ{j}{N}w_j(f(\mathbf{x}_j))^2 - 2K(f) \leq \summ{j}{N}w_j\norminfty{f}^2 - 2K(f), \\
&= V\norminfty{f}^2 - 2K(f).
\end{aligned}
\end{equation*}
Since $ \normtwo{\hybridhyper{f}} \leq \sqrt{ V\norminfty{f}^2 - 2K(f)  }$ and $K(f) >0$, there exist $\tau_1=\tau_1(f) <1$ such that
\begin{eqnarray*}
    \normtwo{\hybridhyper{f}} \leq \sqrt{ V\norminfty{f}^2 - 2K(f)  } = \tau_1 V^{1/2}\norminfty{f}.
\end{eqnarray*}
If $\lambda=0$, then
\begin{equation}\label{equ:bound_filter}
    \|\mathcal{F}_{L,N}f \|_{2} < V^{1/2} \|f\|_{\infty}.
\end{equation}
Then we obtain
\begin{eqnarray*}
     \normtwo{\hybridhyper{f^{\epsilon}} - f } \leq \| \mathcal{H}_L^{\lambda} f^{\epsilon} - \mathcal{F}_{L,N} f^{\epsilon} \|_2 + \| \mathcal{F}_{L,N} f^{\epsilon} -f \|_2 .
\end{eqnarray*}
Next, by \eqref{equ:bound_filter} and $\mathcal{F}_{L,N}q=q$ for all $q\in \mathbb{P}_{\lfloor L/2 \rfloor}(\Omega)$, we have
\begin{eqnarray*}
    \begin{aligned}
        \| \mathcal{F}_{L,N}f^{\epsilon} -f \|_{2} &\leq \| \mathcal{F}_{L,N}(f^{\epsilon} -q^{\ast}) + \mathcal{F}_{L,N}q^{\ast}-f\|_{2}, \\
        & \leq \| \mathcal{F}_{L,N}(f^{\epsilon} -q^{\ast})\|_2 + \|f-q^{\ast}\|_2, \\
        & \leq V^{1/2}\|f^{\epsilon}-q^{\ast}\|_{\infty} +  V^{1/2}\|f-q^{\ast}\|_{\infty}, \\
        & \leq V^{1/2}\|f^{\epsilon}-f\|_{\infty} +  2V^{1/2}E_{\lfloor L/2 \rfloor}(f),
    \end{aligned}
\end{eqnarray*}
where $q^{\ast} \in \mathbb{P}_{\lfloor L/2 \rfloor}(\Omega)$ is the best approximation of $f$ in $\mathbb{P}_{\lfloor L/2 \rfloor}(\Omega)$, the second term on the right side in the third row follows from the Cauchy-Schwarz inequality, which ensures 
\begin{equation*}\|g\|_2=\sqrt{\langle g, g \rangle} \leq \| g\|_{\infty}\sqrt{\langle 1, 1 \rangle}= V^{1/2}\|g\|_{\infty}, \quad\forall g \in \mathcal{C}(\Omega).
\end{equation*} It is obvious that
\begin{eqnarray*}
    \lim_{\lambda \to 0}  \| \mathcal{H}_L^{\lambda} f^{\epsilon} - \mathcal{F}_{L,N} f^{\epsilon} \|_2 =0 \quad \text{and} \quad \lim_{L \to \infty} E_{\lfloor L/2 \rfloor}(f) =0.
\end{eqnarray*}

Thus, we obtain
\begin{eqnarray*}
    \lim_{L \to \infty} \lim_{\lambda \to 0} \normtwo{\hybridhyper{f^{\epsilon}} - f } = V^{1/2}\norminfty{f-f^{\epsilon}}.
\end{eqnarray*}
\end{proof}

\begin{remark}
Although $\lim_{L \to \infty} \normtwo{\mathcal{L}_Lf^{\epsilon}  - f }=  V^{1/2} \| f^{\epsilon} -f  \|_{\infty}$,  we consider hybrid hyperinterpolation as an effective tool to recover the objective function with noise since it possesses the basis selection ability and doubly shrinks less relevant coefficients.
\end{remark}


\section{Numerical experiments}

In this section, we will take as $\Omega$, the unit-sphere $\mathbb{S}^2 \subset \mathbb{R}^3$ and the case of a bivariate domain defined as the union of disks. The main interest for the latter is that, differently the previous example, an orthonormal basis is not theoretically available and must be computed numerically.

In the two regions mentioned above, we have examined {\em{Filtered}}, {\em{Lasso}}, {\em{Hybrid}} as well as the classical {\em{Hyperinterpolation}} to make a contrast. Besides, all the penalty parameters $\{\mu_{\ell}\}_{\ell=1}^d$  in  Lasso and hybrid hyperinterpolations are set to be 1.

In the case of filtered hyperinterpolation we have used as $h: [0,1] \rightarrow [0,1]$ the function
\begin{eqnarray*}
 h(x)=   \begin{cases}
1, &x \in [0, \frac{1}{2}], \\
\sin^2 (\pi \, x), &x \in [\frac{1}{2},1]. 
    \end{cases}
\end{eqnarray*}

Depending on the numerical experiments, we have added noise to the evaluation of a function $f$ on the $N$ nodes $\{ {\mathbf{x}_j}\}_{j=1}^{N}$. In particular we considered
\begin{itemize}
\item {\em{Gaussian noise}} ${\cal{N}}(0,\sigma^2)$ from a normal distribution with mean 0
        and standard deviation {\tt{sigma}}=$\sigma$, implemented via the Matlab command
\begin{center}
{\tt{sigma*randn(N,1)}};
\end{center}
\item {\em{Impulse noise}} ${\cal{I}}(a)$ that takes a uniformly distributed random values in $[-a,a]$ with probability $1/2$ by means of the Matlab command
\begin{center}
{\tt{a*(1-2*rand(N,1))*binornd(1,0.5)}}.
\end{center}
since {\tt{binornd(1,0.5)}} returns an array of random numbers chosen from a binomial distribution with parameters $1$ and $1/2$.
\end{itemize}

{\vspace{0.1cm}}

All Matlab codes used in these experiments are available at a GITHUB homepage \url{https://github.com/alvisesommariva/HYPER24}. All tests were performed by Matlab  R2022a, Update 4.



\subsection{Union of disks}\label{UnionDisk}

In this section we suppose $\Omega$ is the union of $M$ disks in general with different centers $C_k$ and radii $r_k$, i.e. $\Omega=\cup_{k=1}^{M} B(C_k,r_k) \subset {\mathbb{R}}^2$. In {\cite{sommoriva2022cubature}}, it was introduced an algorithm that determine low-cardinality cubature rules of PI-type such that
\begin{eqnarray}\label{UOD}
    \int_{\Omega} p({\mathbf{x}}) \, \text{d}x \text{d}y = \sum_{k=1}^{N_n} w_k p({\mathbf{x}}_k), \quad p\in {\mathbb{P}}_n(\Omega),
\end{eqnarray}
where $n$ is the algebraic degree of exactness fixed by the user. These rules have the property that $N_n \leq d={\mbox{dim}}({\mathbb{P}}_n)$.

Thus, having in mind to provide hyperinterpolants of degree $L$ we will adopt a rule with degree $2L$ and introduce the scalar product
\begin{eqnarray}\label{UOD2}
\langle f,g \rangle=\langle f,g \rangle_N=\sum_{k=1}^{N} w_k f({\mathbf{x}}_k)g({\mathbf{x}}_k), \quad f,g \in {\mathbb{P}}_{L}(\Omega),
\end{eqnarray}
where $N=N_{2L}$.

In particular, the polynomial orthonormal basis of ${\mathbb{P}}_{L}(\Omega)$ in general is not known explicitly and must be computed numerically, following the arguments exposed in \cite{sommariva2017nonstandard}. In detail, let $\{\varphi_k\}_{k=1}^{d}$ be a triangular polynomial basis of ${\mathbb{P}}_n$ and $\mathbf{V}_{\varphi}=(\varphi_j({\mathbf{x}}_i)) \in {\mathbb{R}}^{N \times d}$ be the Vandermonde matrix at the cubature nodes of a formula (\ref{UOD}) of PI-type with algebraic degree of exactness $2L$, as in {\cite{sommoriva2022cubature}}. Let ${\sqrt{\mathbf{W}}} \in {\mathbb{R}}^{N \times N}$be the diagonal matrix whose entries are ${\sqrt{\mathbf{W}}}_{i,i}=\sqrt{w_i}$, $i=1,\ldots,N$. Under these assumptions it can be shown that $\mathbf{V}_{\varphi}$ has full-rank and we can apply $QR$ factorization ${\sqrt{\mathbf{W}}}\mathbf{V}_{\varphi}=\mathbf{QR}$ with $\mathbf{Q} \in {\mathbb{R}}^{N \times d}$ an orthogonal matrix and $\mathbf{R} \in {\mathbb{R}}^{d \times d}$ an upper triangular non-singular matrix. This easily entails that the polynomial basis $(\Phi_1,\ldots,\Phi_d)=(\varphi_1,\ldots,\varphi_d)\mathbf{R}^{-1}$ is orthonormal with respect to the scalar product $\langle f,g \rangle_N$ and consequently by (\ref{UOD2}) to that of $L_2(\Omega)$.
A fundamental aspect is that since $\mathbf{R}^{-1}$ is an upper triangular non-singular matrix, then the basis $(\Phi_1,\ldots,\Phi_d)$ is also triangular and ${\mbox{deg}} (\Phi_k)={\mbox{deg}} (\varphi_k)$, $k=1,\ldots,d$. This result is fundamental to apply filtered hyperinterpolation over $\Omega=\cup_{k=1}^{M} B(C_k,r_k)$.
In our numerical experiments, as triangular basis $\{\varphi_k\}_{k=1}^{d}$, $d=(L + 1)(L + 2)/2$, of ${\mathbb{P}}_L(\Omega)$ we considered the total-degree product Chebyshev basis of the smallest Cartesian rectangle $[a_1,b_1] \times [a_2,b_2]$ containing $\Omega$, with the algebraic degree of exactness lexicographical ordering.

As for the domain, we set $\Omega=\Omega^{(r_1)} \cup \Omega^{(r_2)}$, where $\Omega^{(r_j)}$, $j=1,2$, is the union of $19$ disks with centers $P_k^{(r_j)}=(r_j \cos(\theta_k),r_j \sin(\theta_k))$, where $\theta_k=2k\pi/19$, $k=0,\ldots,18$ and radius equal to $r_j/4$, with $r_1=2$
and $r_2=4$; notice that the set is the disconnected union of two multiply connected unions (see Figure \ref{fig_e1}).

\begin{figure}[!htbp]
  \centering
   {\includegraphics[scale=0.15,clip]{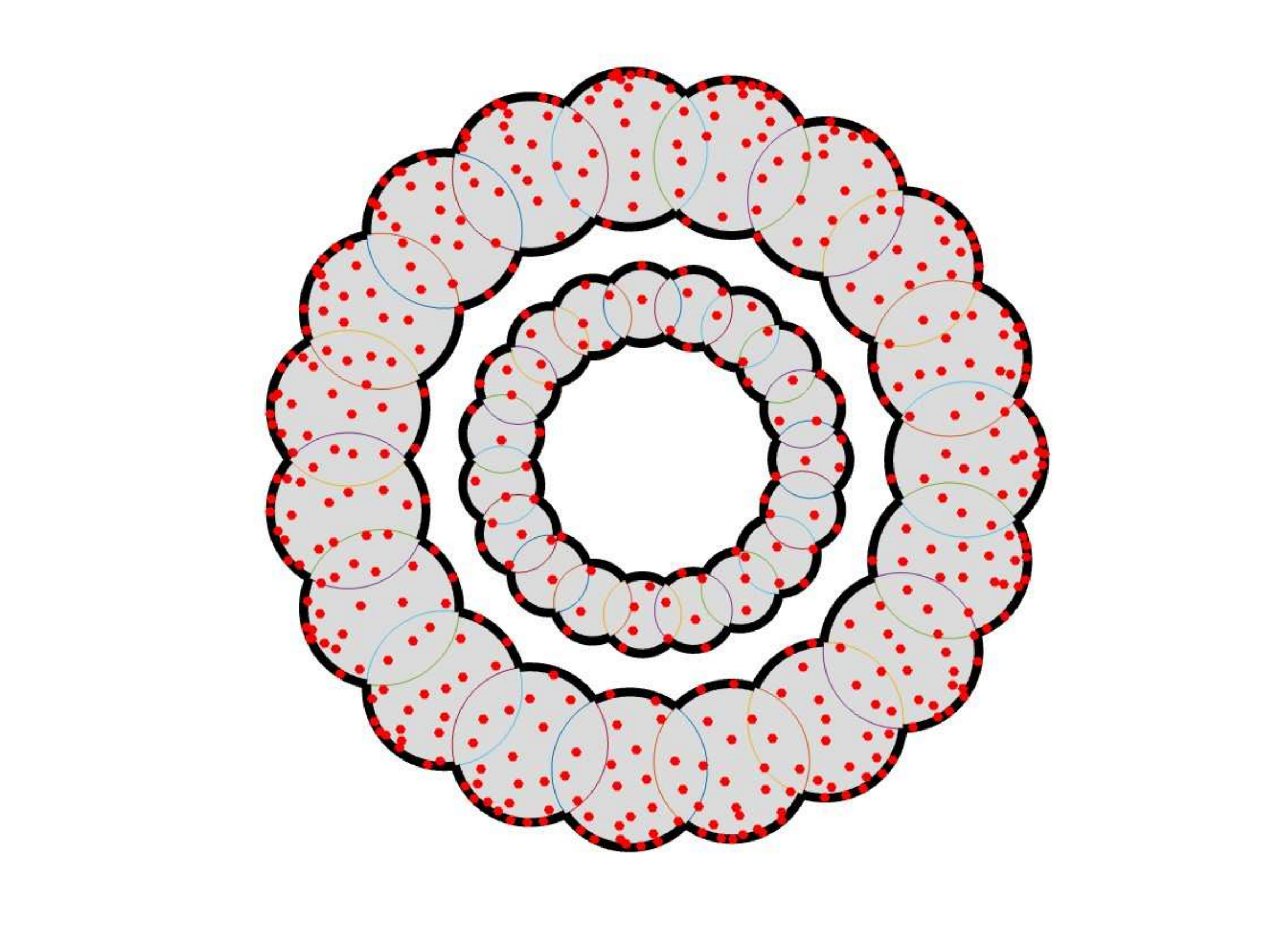}}
 \caption{The domain $\Omega=\Omega^{(r_1)} \cup \Omega^{(r_2)}$ in which we perform our tests. We represent in red the $N=496$ nodes of the cubature rule for algebraic degree of exactness=$30$.}
 \label{fig_e1}
\end{figure}

In order to define all the hyperinterpolants at $L=15$ the required formula must have algebraic degree of exactness $2L$ and consists of $N=496$ nodes that is equal to the dimension of the polynomial space ${\mathbb{P}}_{30}(\Omega)$.

Next, we consider contaminated evaluations of
\begin{eqnarray*}
    f({\mathbf{x}})=(1-(x^2+y^2)) \exp ( x \, \cos(y)),
\end{eqnarray*}
where ${\mathbf{x}}=(x,y)^{\rm{T}}$, at the nodes $\{{\mathbf{x}}_j\}_{j=1}^{N}$,  contaminated by Gaussian noise ($\sigma=0.05$).

In Figure \ref{figure:unionofdisks}, we test the denoising abilities of hyperinterpolation $\mathcal{L}_L f^{\epsilon}$, filtered hyperinterpolation $\mathcal{F}_{L,N} f^{\epsilon}$, and hybrid hyperinterpolation $\mathcal{H}_{L}^{\lambda} f^{\epsilon}$ with the given regularization parameter $\lambda$ to reach its smallest $L_{\infty}$ error. As illustrated in Figure \ref{figure:unionofdisks2} for hybrid hyperinterpolation, the smallest $L_2$ error of 0.016905 is obtained for $\lambda=0.004809$, while the smallest $L_{\infty}$ error of 0.039246 is obtained for $\lambda=0.006801$. From Table \ref{union_disk_table}, it is clear that hybrid hyperinterpolation makes a compromise between the sparsity and smallest approximation errors in $L_2$ and $L_{\infty}$ norms. Although hybrid hyperinterpolation is less sparse than Lasso hyperinterpolation, it achieves a more robust denoising effect.


Numerical tests show the favourable ratio between the sparsity and $L_2$ errors of the hybrid hyperinterpolants.

\begin{table}
\begin{center}
  \begin{tabular}{|c||c| c  |c||c| c  |c|}
    \hline
    \multirow{2}{*}{} &
      \multicolumn{1}{c |}{$\lambda$ } &
      \multicolumn{1}{c |}{smallest $L_2$ error} &
      \multicolumn{1}{c ||}{sparsity } &
      \multicolumn{1}{c |}{$\lambda$ } &
      \multicolumn{1}{c |}{smallest $L_{\infty}$ error}&
      \multicolumn{1}{c |}{sparsity } \\
    \hline

{\em{Hyperint.}} & $ - $ & $ 0.045506 $ & $ 136 $ & $ - $ & $0.118451$ & $136$ \\

{\em{Filtered}} & $ - $ & $ 0.031313 $ & $ 120 $ & $ - $ & $0.104317$ & $120$ \\

{\em{Lasso}} & $ 0.005921 $ & $ 0.019482 $ & $ 30 $ & $ 0.007289 $ & $0.044901$ & $16$ \\

{\em{Hybrid}} & $ 0.004809 $ & $ 0.016905 $ & $ 37 $ & $ 0.006801 $ & $0.039246$ & $21$ \\

    \hline
  \end{tabular}
  \caption{Smallest approximation errors in the $L_2$ and $L_{\infty}$ norms and the sparsity of variants of hyperinterpolation coefficients of noisy versions of $f(x,y)=(1-(x^2+y^2)) \exp ( x \, \cos(y))$  over the domain  union of disks, contaminated with Gaussian noise ($\sigma =0.05$). The hyperinterpolation degree is $15$ and we list the $\lambda$ that provide the best results for these norms for Lasso and Hybrid hyperinterpolation are indicated in the second and fifth column.}
  \label{union_disk_table}
  \end{center}
\end{table}

\begin{figure}[htbp]
    \centering
    \includegraphics[width=\textwidth]{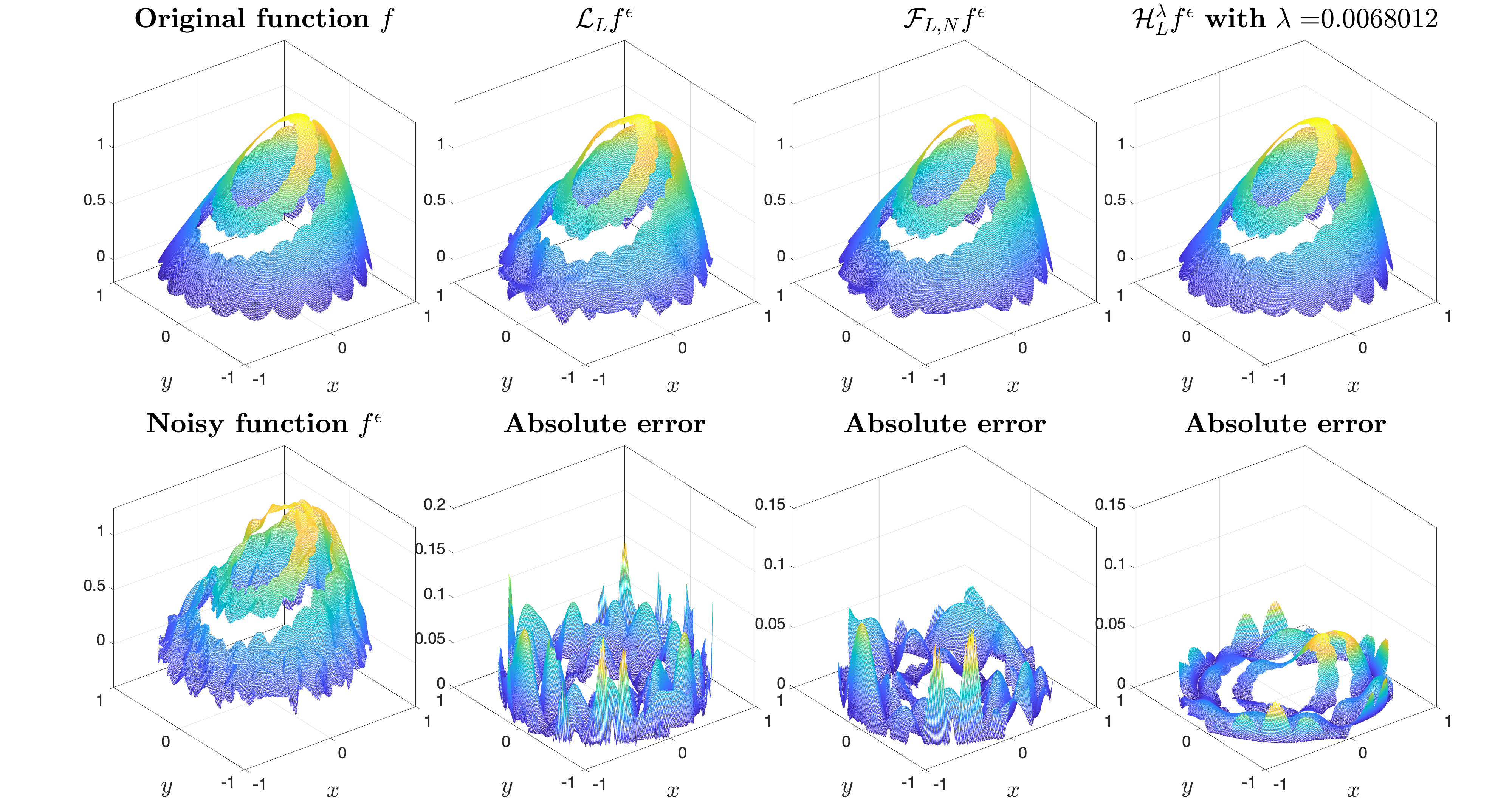}
    \caption{Approximate  $f(x,y)=(1-(x^2+y^2)) \exp ( x \, \cos(y))$, perturbed by Gaussian noise ($\sigma=0.05$), over the union of disks, via
    hyperinterpolation $\mathcal{L}_L f^{\epsilon}$,
    filtered hyperinterpolation $\mathcal{F}_{L,N}f^{\epsilon}$, and
    hybrid hyperinterpolation $\mathcal{H}_{L}^{\lambda} f^{\epsilon}$ with $L=15$.}
    \label{figure:unionofdisks}
\end{figure}

\begin{figure}[htbp]
    \centering
    \includegraphics[width=\textwidth]{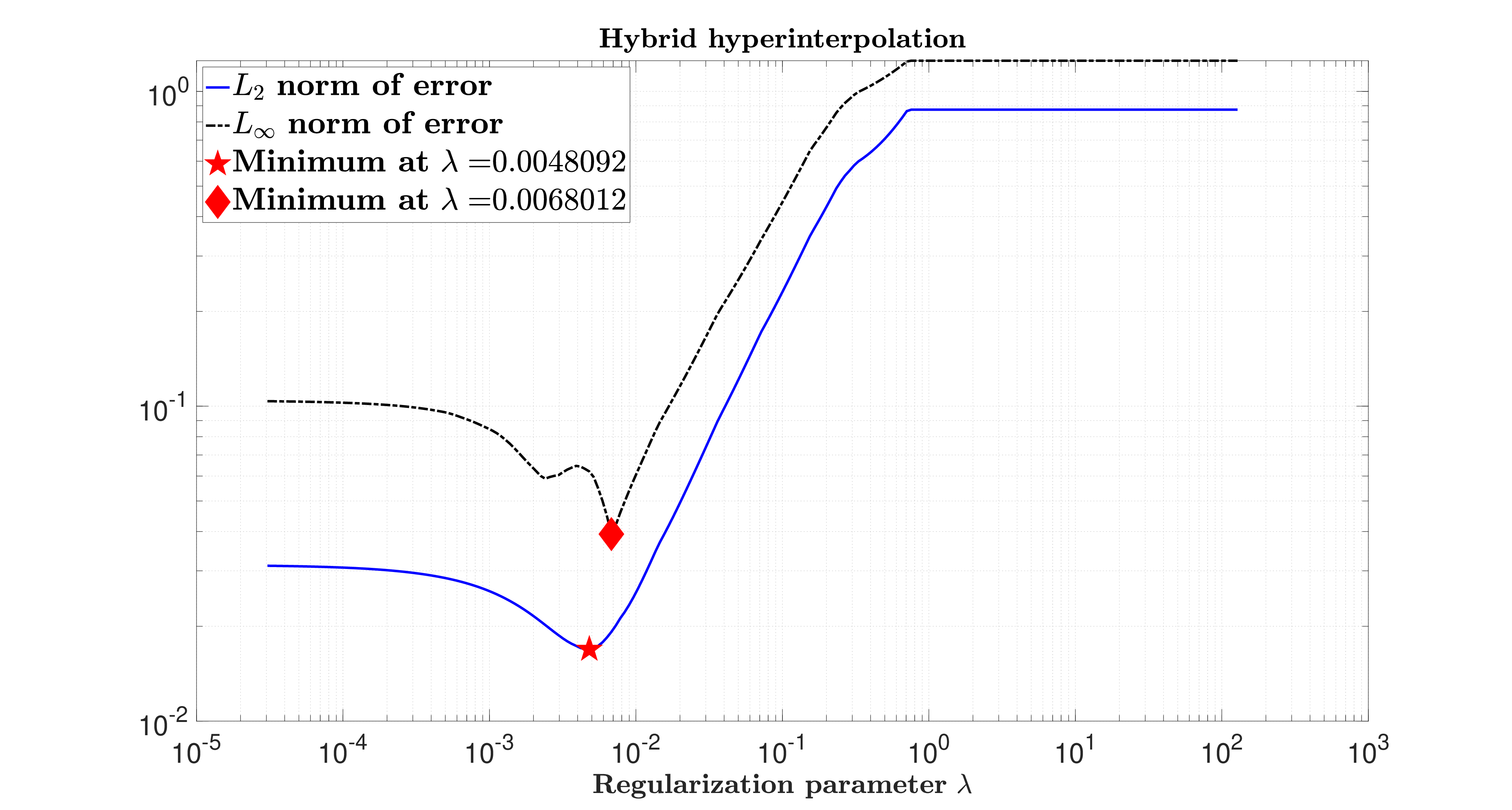}
    \caption{The choices of regularization parameter $\lambda$   for hybrid hyperinterpolation $\mathcal{H}_L^{\lambda} f^{\epsilon}$ with $L=15$ approximating
    $f(x,y)=(1-(x^2+y^2)) \exp ( x \, \cos(y))$, perturbed by Gaussian noise ($\sigma=0.05$), over the union of disks.} 
    \label{figure:unionofdisks2}
\end{figure}

\subsection{The unit-sphere \texorpdfstring{${\mathbb{S}}^2$}{}} 

As the second domain we set as $\Omega$ the unit-sphere, i.e. ${\mathbb{S}}^2=\{{\mathbf{x}}: \|{\mathbf{x}}\|_2=1\}$ (see e.g. {\cite{hesse2010numerical}}). Between the most studied families of quadrature rules on ${\mathbb{S}}^2$ with respect to the surface measure $\measure$  (see e.g. {\cite{hesse2010numerical}}), there are the so called {\em{spherical $t$-designs}}, introduced by Delsarte, Goethals and Seidel in 1977 (see, e.g. the pioneering work {\cite{delsarte1977spherical}} and for recent advances {\cite{womersley2018efficient}} and references therein). A pointset $\{{\mathbf{x}}_1,\ldots,{\mathbf{x}}_N\} \subset {\mathbb{S}}^2$ is a spherical $t$-design if it satisfies
\begin{eqnarray*}
    {\frac{1}{N}} \sum_{j=1}^N p({\mathbf{x}}_j) = \int_{{\mathbb{S}}^2} p({\mathbf{x}}) \text{d} {\omega}({\mathbf{x}}), {\mbox{  }} \forall p \in {\mathbb{P}}_t({\mathbb{S}}^2).
\end{eqnarray*}
In this work we shall consider in particular an {\em{efficient spherical design}} ${\cal{X}}_N$, between those proposed in {\cite{womersley2018efficient}}. In our tests, since we intend to compute several type of hyperinterpolants of total degree at most $L=15$, it is necessary to adopt a rule with algebraic degree of exactness equal to $2L=30$, consisting of $N=482$ points.
We thus have that for any $f,g \in {\mathbb{P}}_L({\mathbb{S}}^2)$
\begin{eqnarray*}
    \int_{{\mathbb{S}}^2} f({\mathbf{x}})g({\mathbf{x}}) \text{d} {\omega}({\mathbf{x}})=  \langle f,g \rangle=\langle f,g \rangle_{961}:=\sum_{j=1}^{961} w_j f({\mathbf{x}}_j) g({\mathbf{x}}_j).
\end{eqnarray*}

We adopt the so called {\em{spherical harmonics}} {\cite{atkinson2012spherical}} as the triangular and orthonormal polynomial basis.  The fact that the basis is triangular allows to determine easily the total degree of each polynomial, making possible the application of filtered hyperinterpolation.

Following {\cite{an2010well}}, we have considered the function
\begin{eqnarray*}
    f({\mathbf{x}})=  {\frac{1}{3}} \sum_{i=1}^6 \Phi_2(\| {\mathbf{z}}_i - {\mathbf{x}}\|_2),
\end{eqnarray*}
where $\Phi_2(r):={\tilde{\Phi}}_2(r/\delta_2)$, in which ${\tilde{\Phi}}_2$ is the $\mathcal{C}^6$ compactly supported of minimal degree Wendland function defined as
\begin{eqnarray*}
    {\tilde{\Phi}}_2(r):=(\max\{1-r,0\})^6 (35 r^2 + 18r +3)
\end{eqnarray*}
and $\delta_2=\frac{9 \Gamma(5/2)}{2\Gamma(3)} $. We perturb $f$ by adding Gaussian noise ${\cal{N}}(0,\sigma^2)$ with standard deviation $\sigma=0.02$ and single impulse noise ${\cal{I}}(a)$ relatively to the level $a=0.02$.

In Figure \ref{figure:sphere1}, we show the the denoising abilities of hyperinterpolation $\mathcal{L}_L f^{\epsilon}$, filtered hyperinterpolation $\mathcal{F}_{L,N} f^{\epsilon}$, and hybrid hyperinterpolation $\mathcal{H}_{L}^{\lambda} f^{\epsilon}$ with the given regularization parameter $\lambda$ to reach its smallest $L_{\infty}$ error. As illustrated in Figure \ref{figure:sphere2} for hybrid hyperinterpolation, the smallest $L_2$ error of 0.012466 is obtained for $\lambda=0.005154$, while the smallest $L_{\infty}$ error of 0.008571 is obtained for $\lambda=0.006801$. From Table \ref{sphere_table}, it is evident that hybrid hyperinterpolation strikes a balance between sparsity and minimizing approximation errors in both $L_2$ and $L_{\infty}$ norms. While it is less sparse than Lasso hyperinterpolation, hybrid hyperinterpolation provides a more effective denoising outcome.

Consistently with tests in previous domains, it is remarkable  that hybrid hyperinterpolation can reach both smallest $L_2$ and $L_{\infty}$ errors compared with hyperinterpolation, filtered hyperinterpolation, and Lasso hyperinterpolation. 

\begin{table}
\begin{center}
   \begin{tabular}{|c||c| c  |c||c| c  |c|}
    \hline
    \multirow{2}{*}{} &
      \multicolumn{1}{c |}{$\lambda$ } &
      \multicolumn{1}{c |}{smallest $L_2$ error} &
      \multicolumn{1}{c ||}{sparsity } &
      \multicolumn{1}{c |}{$\lambda$ } &
      \multicolumn{1}{c |}{smallest $L_{\infty}$ error}&
      \multicolumn{1}{c |}{sparsity } \\
    \hline

{\em{Hyperint.}} & $ - $ & $ 0.052101 $ & $ 256 $ & $ - $ & $0.058726$ & $256$ \\

{\em{Filtered}} & $ - $ & $ 0.036648 $ & $ 225 $ & $ - $ & $0.041651$ & $225$ \\

{\em{Lasso}} & $ 0.005921 $ & $ 0.013738 $ & $ 26 $ & $ 0.007289 $ & $0.009675$ & $13$ \\

{\em{Hybrid}} & $ 0.005154 $ & $ 0.012466 $ & $ 31 $ & $ 0.006801 $ & $0.008571$ & $17$ \\

    \hline
  \end{tabular}
  \caption{Smallest approximation errors in the $L_2$ and $L_{\infty}$ norms and the sparsity of variants of hyperinterpolation coefficients of noisy versions of $f(x) ={\frac{1}{3}} \sum_{i=1}^6 \Phi_2(\| {\mathbf{z}}_i - {\mathbf{x}}\|_2)$, with  impulse noise ($a=0.02$) and  Gaussian noise ($\sigma=0.02$) added on, over the unit-sphere ${\mathbb{S}}^2$. We list the $\lambda$ that provide the best results for these norms for Lasso and Hybrid hyperinterpolation are indicated in the second and fifth column}
  \label{sphere_table}
  \end{center}
\end{table}

\begin{figure}[htbp]
    \centering
    \includegraphics[width=\textwidth]{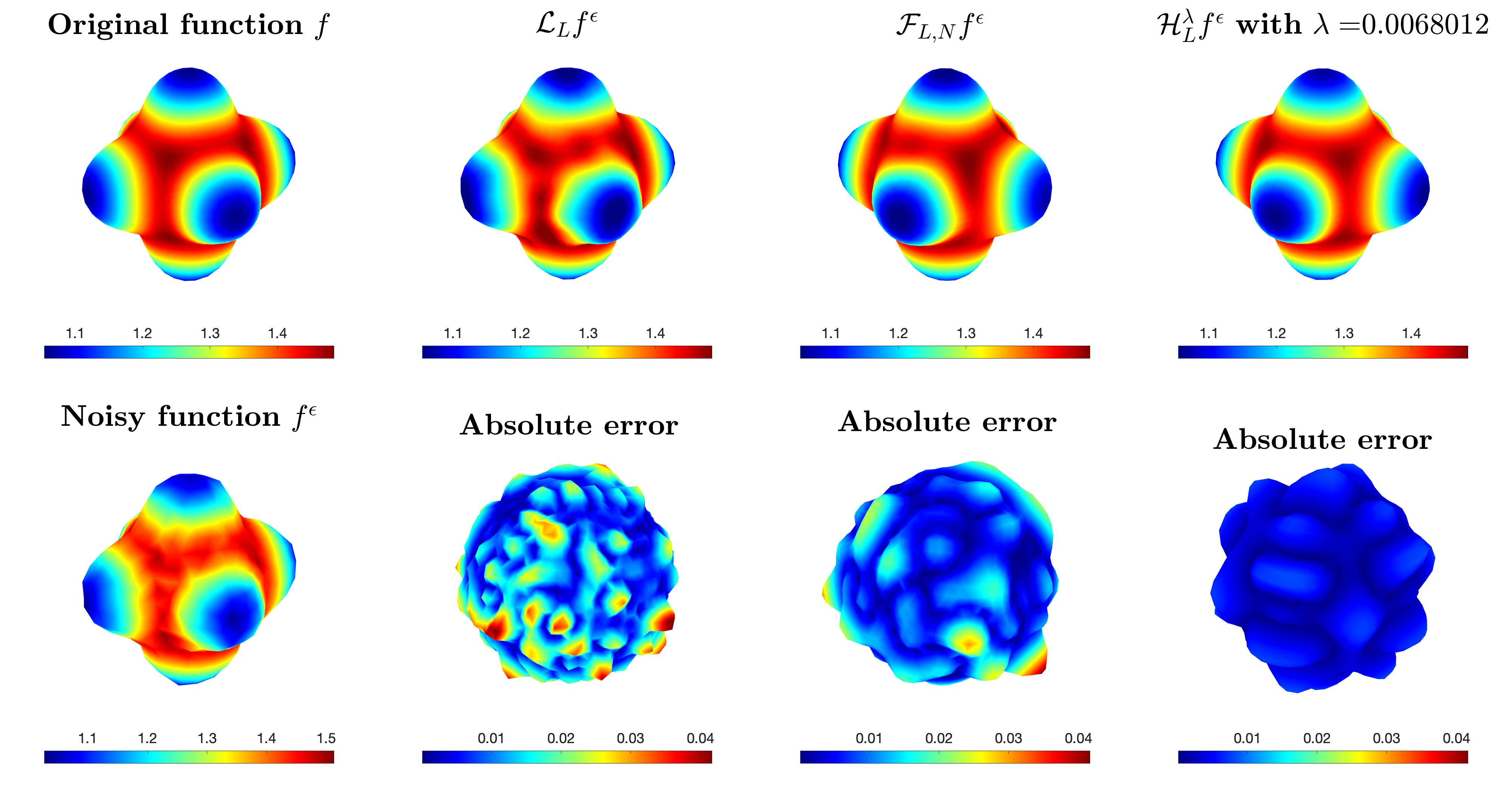}
    \caption{Approximate  $f(x) ={\frac{1}{3}} \sum_{i=1}^6 \Phi_2(\| {\mathbf{z}}_i - {\mathbf{x}}\|_2)$, perturbed by impulse noise ($a=0.02$) and Gaussian noise ($\sigma=0.02$), over the unit-sphere ${\mathbb{S}}^2$, via
    hyperinterpolation $\mathcal{L}_L f^{\epsilon}$,
    filtered hyperinterpolation $\mathcal{F}_{L,N}f^{\epsilon}$, and
    hybrid hyperinterpolation $\mathcal{H}_{L}^{\lambda} f^{\epsilon}$ with $L=15$.}
    \label{figure:sphere1}
\end{figure}

\begin{figure}[htbp]
    \centering
    \includegraphics[width=\textwidth]{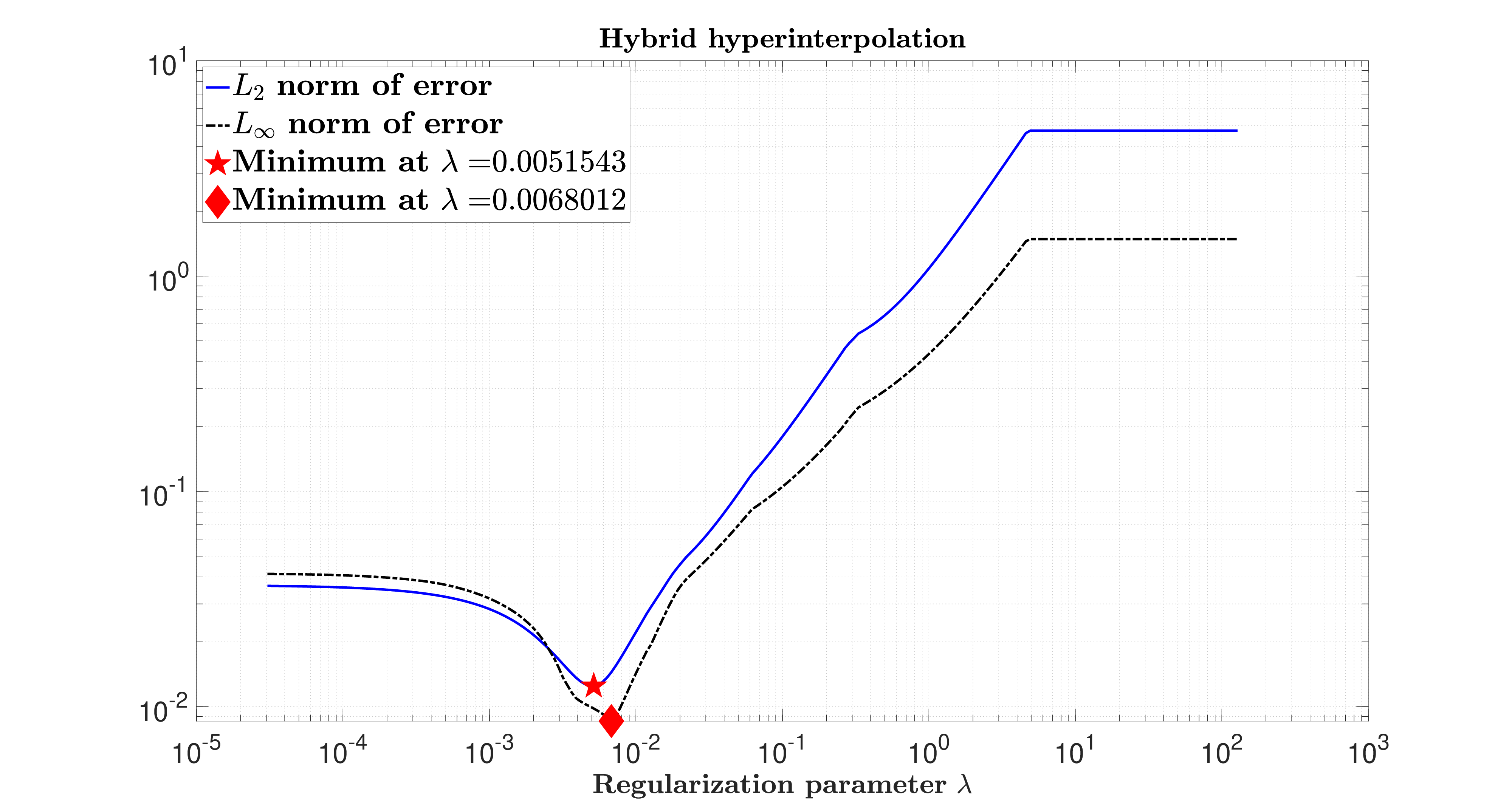}
    \caption{The choices of regularization parameter $\lambda$  for hybrid hyperinterpolation   $\mathcal{H}_L^{\lambda} f^{\epsilon}$ with $L=15$ approximating
    $f(x) ={\frac{1}{3}} \sum_{i=1}^6 \Phi_2(\| {\mathbf{z}}_i - {\mathbf{x}}\|_2)$, perturbed by impulse noise ($a=0.02$) and Gaussian noise ($\sigma=0.02$), over the unit-sphere ${\mathbb{S}}^2$.}
    \label{figure:sphere2}
\end{figure}

\section{Conclusions and outlook}


In this study, we have demonstrated the efficacy and resilience of hybrid hyperinterpolation in recovering noisy functions across the unit sphere and the union of disks. Combining Lasso and filtered hyperinterpolations, hybrid hyperinterpolation offers denoising and basis selection capabilities while doubly shrinking less relevant coefficients. Our numerical experiments confirm the enhanced and robust denoising performance of hybrid hyperinterpolation over the unit-sphere  and the unions of disks. 


Moving forward, there are several promising avenues for further exploration. First, while hybrid hyperinterpolation currently employs an $\ell^2+\ell_1$ regularization scheme, exploring other regularization terms could yield valuable insights. For instance, the springback penalty ($\|x\|_1 - \alpha/2\|x\|_2^2$ with $\alpha >0$ being a model parameter) proposed in recent literature \cite{an2022springback} has shown promise in robust signal recovery.


Second, the choice of filter function $h(\cdot)$ offers various possibilities, and exploring alternative forms could enhance the shrinkage process for both relevant and less relevant coefficients.


Third, recent advancements suggest that the required exactness degree in hyperinterpolation can be relaxed or bypassed \cite{an2022bypassing, an2022quadrature}, opening avenues for further investigation into hyperinterpolation and its variants under this novel framework.


Last but far from the least, exploring the application of hyperinterpolation variants to regions defined by piecewise rational functions, such as curvilinear polygons, presents an intriguing avenue for future research, especially with the advent of low cardinality PI-type algebraic cubature rules \cite{sommoriva2023lowcadrinality}. Hyperinterpolation on the square \cite{caliari2007hyperinterpolation}, on the disk \cite{atkinson2009norm}, in the cube \cite{caliari2008hyperinterpolation},on the sphere \cite{dai2006generalized,LeGia2001uniform,pieper2009vector,reimer2000hyperinterpolation}, over nonstandard planar regions \cite{sommariva2017nonstandard} and spherical triangles \cite{sommariva2021sphericaltri} were also considered in the past.

\section*{Acknowledgements}
The first author (C. An) of the research  is partially supported by  National Natural Science Foundation of China (No.12371099). Work partially supported by the DOR funds of the University of Padova, and by the INdAM-GNCS 2024 Projects “Kernel and polynomial methods for approximation and integration: theory and application software” (A. Sommariva).  This research has been accomplished within the RITA ``Research ITalian network on Approximation", the SIMAI Activity Group ANA\&A, and the UMI Group TAA ``Approximation Theory and Applications" (A. Sommariva).

\section*{Conflict of interest Statement} The authors declare that they have no conflict of interest.
\section*{Ethical statement} The submitted work is original and has not been published elsewhere in any form or language. Our investigations do not involve animal experimentation or human participants.
\section*{Informed consent} On behalf of the authors, Prof. Alvise Sommariva shall be communicating the manuscript.

\newpage
\bibliographystyle{siamplain}
\bibliography{ref_HybridHyper}
\clearpage

\end{document}